\documentclass[12pt, reqno]{amsart}
\usepackage{amsmath, amsthm, amscd, amsfonts, amssymb, graphicx }
\usepackage{mathtools}
\usepackage{blkarray, bigstrut} %
\usepackage[margin=1in]{geometry}
\newtheorem{theorem}{Theorem}[section]
\newtheorem{lemma}[theorem]{Lemma}

\newtheorem{corollary}[theorem]{Corollary}
\theoremstyle{definition}

\newtheorem{example}[theorem]{Example}

\theoremstyle{remark}
\newtheorem{remark}[theorem]{Remark}
\numberwithin{equation}{section}

\begin{document}

	\setcounter{page}{1}
	\title[WEYL TYPE THEOREMS FOR HYPERCYCLIC, SUPERCYCLIC, AND TOEPLITZ OPERATORS]{ WEYL TYPE THEOREMS FOR HYPERCYCLIC, SUPERCYCLIC, AND TOEPLITZ OPERATORS}
	\author{ simi thomas\textsuperscript{1}, thankarajan prasad\textsuperscript{2}   and shery fernandez\textsuperscript{3} }

	
		\let\thefootnote\relax\footnote{\textsuperscript{1}Department of Mathematics, Cochin University of Science and Technology, Kerala, India -682022, \textbf{ simimariumthomas@cusat.ac.in}.}
	
	\let\thefootnote\relax\footnote{\textsuperscript{2}Department of Mathematics, University of Calicut,
		Kerala, India -673635, \textbf{prasadvalapil@gmail.com}.}
	\let\thefootnote\relax\footnote{\textsuperscript{3}Department of Mathematics, Cochin University of Science and Technology, Kerala, India -682022, \textbf{sheryfernandez@cusat.ac.in}.}
	
	
	
\begin{abstract}
 In this paper, we study property $(UW_E)$ for hypercyclic and supercyclic operators. The stability of variants of Weyl type theorems under compact perturbations for Toeplitz operators on the Bergman space is also studied. We also provide some examples of Toeplitz operators satisfying Weyl type theorems on the Bergman space and the harmonic Bergman space.
\end{abstract}\maketitle
\section{Introduction}
Hermann Weyl \cite{Weyl} made a significant observation about how the spectrum and the Weyl spectrum of bounded self-adjoint operators on separable Hilbert spaces are related. He identified that if one removes the Weyl spectrum from the spectrum, the remaining points are isolated eigenvalues with finite multiplicities. This result is called Weyl's theorem. Coburn \cite{cob} verified this for non-normal operators, including hyponormal operators. Berberian \cite{berb} extended Weyl's theorem to operators that are not necessarily normal. Later, Weyl's theorem was extended to several classes of operators (see \cite{aiena,cob,mec,qiu}). Variants of Weyl's theorem called Weyl type theorems have been introduced and studied by many authors (see \cite{aiena,berk,jia,liu,sana,yang}).
\par Property $(UW_E)$ is one of the variants of Weyl type theorems introduced by Berkani and Kachad \cite{berk}. Chenhui and Cao \cite{sun} characterized property $(UW_E)$ by using topological uniform descent. Qiu, Sinan and Cao studied property $(UW_E)$ for operators and operator matrices \cite{qiu}. In \cite{prasad}, the authors studied property $(UW_E)$ for operators and its stability under compact perturbations. \par In this paper, we study property $(UW_E)$ for hypercyclic and supercyclic operators and give some examples of hypercyclic and supercyclic operators that do not satisfy property $(UW_E)$. We also give examples of operators satisfying property $(UW_E)$ that are neither hypercyclic nor supercyclic. Moreover, we study the spectral properties of Toeplitz operators on the Bergman space that satisfy property $(UW_E)$. We show that some Weyl type theorems are invariant under compact perturbations for Toeplitz operators on the Bergman space with continuous symbols under certain conditions. We verify that there exists a non-hyponormal Toeplitz operator on the  harmonic Bergman space that satisfies property $(UW_E)$. 
\section{Preliminaries}
Let $\mathcal{H}$ be an infinite dimensional complex separable Hilbert space, $\mathcal {B(H)}$ be the algebra of all bounded linear operators on $\mathcal{H}$, and let $\mathcal {K(H)}$ denote the ideal of all compact operators on $\mathcal{H}$. The \textit{kernel}, the \textit{range}, the \textit{nullity}, and the \textit{defect} of an operator $T\in \mathcal {B(H)}$ are denoted by $\mathcal{N}(T), \mathcal{R}(T), \alpha(T)$, and $\beta(T)$, respectively. An operator $T\in \mathcal {B(H)}$ is said to be an \textit{upper semi-Fredholm} (resp., \textit{lower semi-Fredholm}) operator if $\alpha (T)< \infty $ (resp., $\beta(T)< \infty$) and $\mathcal{R}(T) $ is closed. Let $$U_{SF}(\mathcal{H})=\lbrace T \in \mathcal{B(H)}: T \text{ is upper semi-Fredholm} \rbrace$$ and $$L_{SF}(\mathcal{H})=\lbrace T \in \mathcal {B(H)}: T \text{ is lower semi-Fredholm} \rbrace $$ be the collections of all upper and lower semi-Fredholm operators, respectively. An operator $T \in \mathcal {B(H)}$ is called \textit{semi-Fredholm}, $  T \in SF(\mathcal{H}) $, if $T \in U_{SF}(\mathcal{H}) \cup L_{SF}(\mathcal{H})$. Let $F(\mathcal{H})= U_{SF}(\mathcal{H}) \cap L_{SF}(\mathcal{H})$ denote the collection of all \textit{Fredholm operators}. The \textit{index} of an operator $T \in SF(\mathcal{H})$ is given by $$ind(T)= \alpha(T)-\beta(T).$$ An operator $T \in \mathcal {B(H)}$ is \textit{Weyl} if it is Fredholm of index zero. Let $$SF_+(\mathcal{H})  = \lbrace T \in SF(\mathcal{H}): ind(T)\leq0 \rbrace$$ and $$SF_-(\mathcal{H})  = \lbrace T \in SF(\mathcal{H}): ind(T)\geq 0 \rbrace$$ be the collection of all \textit{upper semi-Weyl operators} and the collection of all \textit{lower semi-Weyl operators}, respectively. The \textit{ascent} and \textit{descent} of an operator $T\in \mathcal {B(H)}$ are given by $$\text{asc}(T)= \text{inf} \lbrace n\in \mathbb{N} \cup \lbrace\infty \rbrace :\mathcal{N}(T^n)=\mathcal{N}(T^{n+1})\rbrace$$ ~\text{and} $$\text{dsc}(T)=\text{inf} \lbrace n\in \mathbb{N} \cup \lbrace\infty \rbrace : \mathcal{R}(T^n)=\mathcal{R}(T^{n+1})\rbrace.$$ An operator $T$ is said to be \textit{Browder} if $T \in F(\mathcal{H})$ and $\text{asc}(T)=\text{dsc}(T)< \infty$. We denote the spectrum and the approximate spectrum of $T \in \mathcal{B(H)}$ as $\sigma(T)$ and $\sigma_a(T)$, respectively. The \textit{essential spectrum} $\sigma_{e}(T)$, the \textit{Wolf spectrum} $\sigma_{lre}(T)$, the \textit{Weyl spectrum} $\sigma_w(T)$, the \textit{upper semi-Weyl spectrum} $\sigma_{uw}(T)$, and the \textit{Browder spectrum} $\sigma_b(T)$ of $T \in \mathcal {B(H)}$ are given by
\begin{eqnarray*}
	\sigma_e(T) & = & \lbrace \lambda \in \mathbb{C} : T- \lambda I \notin F(\mathcal{H}) \rbrace, \\
	\sigma_{lre}(T) &=& \lbrace \lambda \in \mathbb{C} : T- \lambda I \notin SF(\mathcal{H}) \rbrace,\\
	\sigma_w(T) &=& \lbrace \lambda \in \mathbb{C}  : T-\lambda I \text{\,\,\,is not Weyl} \rbrace,\\
	\sigma_{uw}(T) &=& \lbrace \lambda \in \mathbb{C}  : T-\lambda I \notin SF_+(\mathcal{H}) \rbrace,\\
	\sigma_{b}(T) &=& \lbrace \lambda \in \mathbb{C}  : T-\lambda I \text{\,\,\,is not Browder} \rbrace.
	\end{eqnarray*}
The \textit{semi-Fredholm domain} of $T$ is denoted by $\rho_{SF}(T)= \mathbb{C} \setminus \sigma_{lre}(T)$.
\begin{eqnarray*}
	\rho_{SF}^+(T)&:=&\lbrace \lambda \in \rho_{SF}(T) : ind (T-\lambda I) >0\rbrace,\\
	\rho_{SF}^-(T)&:=&\lbrace \lambda \in \rho_{SF}(T) : ind(T-\lambda I) <0\rbrace, \\
	\rho_{SF}^0(T)&:=&\lbrace \lambda \in \rho_{SF}(T) : ind(T-\lambda I) =0\rbrace.
\end{eqnarray*}
It is evident that $\sigma_{w}(T)= \mathbb{C}\setminus \rho_{SF}^0(T)$ and $\rho_{SF}(T)=\rho_{SF}^0(T) \cup \rho_{SF}^+(T) \cup \rho_{SF}^-(T)$.  Let $\text{iso}~\sigma(T)$ denote the \textit{isolated points} of $\sigma(T)$.
We denote 
\begin{eqnarray*}
	E(T)&:=&\lbrace \lambda \in \text{iso}~\sigma(T) :\alpha (T-\lambda I) >0\rbrace~\text{and}\\
	E_0(T)&:=&\lbrace \lambda \in \text{iso}~\sigma(T) :0<\alpha (T-\lambda I) <\infty\rbrace\\
	P_{00}(T)&:=&\lbrace \lambda \in \sigma(T) : T-\lambda I \text{\,\,\,is Browder}\rbrace.
\end{eqnarray*}
The Weyl spectrum admits the following characterization in terms of spectra of compact perturbations:
$$\sigma_{w}(T)=\bigcap_{K \in \mathcal{K}(\mathcal{H})}\sigma(T+K) \quad \text{\cite{cob}}.$$
An operator $T \in \mathcal{B(H)}$ has the \emph{single valued extension property} at $\lambda_0 \in \mathbb{C}$ (SVEP at $\lambda_0$) if, for every open disc $D$ centered at $\lambda_0$, the only analytic function $f : D \rightarrow \mathcal{H}$ satisfying 
\begin{equation}
(T-\lambda I)f(\lambda) = 0 \quad \text{for all } \lambda \in D
\end{equation}
is the zero function $f \equiv 0$ \cite{aiena}. If the \textit{adjoint} $T^*$ of $T$ has SVEP, then $\sigma(T)=\sigma_a(T)$ \cite{aiena}.
Let $x \in \mathcal{H}$ and $T \in \mathcal{B(H)}$. Then the \textit{orbit} of $x$ with respect to the operator $T$, $\mathrm{Orb}$$(T:x)$ is defined as 
	\begin{center}
		$\mathrm{Orb}$$(T:x)=\lbrace T^n(x): n=0,1,2,3, \ldots\rbrace$.
	\end{center}
     An operator $T$ is said to be \textit{hypercyclic} if there exists some $x \in \mathcal{H}$ such that $\mathrm{Orb}$$(T:x)$ is dense in $\mathcal{H}$. Such a vector $x$ is called a \textit{hypercyclic vector} for $T$. If there exists a vector $x\in \mathcal{H}$ such that 
\begin{center}
	$\mathbb{C} \mathrm{Orb}$$(T:x)=\lbrace \lambda T^n(x): n=0,1,2,3, \ldots, \lambda \in \mathbb{C}\rbrace$
\end{center} is dense in $\mathcal{H}$, then $T$ is called a \textit{supercyclic operator} and $x$ is a \textit{supercyclic vector} for $T$.  Kitai \cite{kita} studied several fundamental results related to the theory of hypercyclic and supercyclic operators.\\
According to Coburn \cite{cob}, an operator $T \in \mathcal{B(H)}$ is said to satisfy Weyl's theorem if
$$\sigma(T)\setminus\sigma_{w}(T)=E_0(T).$$
An operator $T$ is said to satisfy Browder’s theorem if $\sigma_w(T)=\sigma_{b}(T)$. Berkani and  Kachad \cite{berk} introduced and studied property $(UW_E)$, a variant of Weyl type theorem. An operator $T \in \mathcal{B(H)}$ is said to satisfy the property $(UW_E)$ if $\sigma_a(T)\setminus\sigma_{uw}(T)=E(T)$ and property $(w)$, if $\sigma_a(T)\setminus\sigma_{uw}(T)=E_0(T)$ \cite{rako}. Recall that $T \in \mathcal{B(H)}$ satisfies property $(UW_E)$, then $E(T)=E_0(T)$ \cite{berk}. Observe that 
 \begin{center}
 	Property $(UW_E) \Longrightarrow $ Property $(w) \Longrightarrow$ Weyl's theorem $\Longrightarrow$ Browders theorem.
 \end{center}
 An operator $T \in \mathcal{B(H)}$ is said to satisfy property $(V_E)$ if $\sigma(T)\setminus\sigma_{uw}(T)=E(T)$. It is clear that $T \in \mathcal{B(H)}$ satisfies property $(V_E)$ if and only if $T$ satisfies property $(UW_E)$ and $\sigma(T)=\sigma_a(T)$ \cite{sana}. An operator $T \in \mathcal{B(H)}$ is said to satisfy property $(W_E)$ if $\sigma(T)\setminus\sigma_{w}(T)=E(T)$ \cite{berk}. We refer to \cite{jia,rako} for details about variants of Weyl type theorems such as $ a$-Weyl's theorem and $ a$-Browder's theorem.
 \begin{remark}\label{r1}
  Let $T \in \mathcal {B(H)}$. Then $T$ satisfies property $(UW_E)$ if and only if $T$ satisfies Weyl's theorem,  $\sigma_{w}(T)\setminus\sigma_{uw}(T)=\sigma(T)\setminus\sigma_{a}(T)$, and $E(T)=E_{0}(T)$.
\end{remark}
  
 Variants of Weyl type theorems for hypercyclic and supercyclic operators have been extensively researched (see, \cite{dug,elb,herr,hild,liu}). In the next section, we study property $(UW_E)$ for hypercyclic and supercyclic operators.
\section{Property $(UW_E)$ for hypercyclic and supercyclic operators}
We begin this section with a characterization theorem for hypercyclic and supercyclic operators that satisfy property $(UW_E)$.
\begin{theorem} \label{th5}
\begin{enumerate}
    \item  Let $T \in \mathcal{HP}$. Then $T$ satisfies property $(UW_E)$ if and only if $E(T)=\emptyset$.
    \item Let $T \in \mathcal{SP}$ and $E(T) \subseteq E_0(T^*)$. Then $T$ satisfies property $(UW_E)$ if and only if there is an $\alpha \in \mathbb{C} \setminus \sigma_b(T)$ such that $E(T) \subseteq \lbrace \alpha \rbrace$. 
\end{enumerate}
\end{theorem}
\begin{proof}
    \begin{enumerate}
\item Suppose $T$ is hypercyclic and satisfies property $(UW_E)$. Then by Remark \ref{r1}, $T$ satisfies Weyl's theorem and $E(T) = E_0(T)$. Since $T$  satisfies
 Weyl’s theorem, it follows from \cite[Theorem 1.5]{dug} that 
 \begin{center}
     $P_{00}(T)=E(T) = E_0(T)$.
 \end{center}
 This shows that $E_0(T^*)=\emptyset$ and hence 
 \begin{center}
     $E(T) = E_0(T)=P_{00}(T)=P_{00}(T^*)=E_0(T^*)=\emptyset$.
 \end{center}
 Conversely, assume that $T \in \mathcal{HP}$ with $E(T)=\emptyset$. This implies that $E_0(T^*)=\emptyset$. Then from \cite[Proposition 1.2]{dug}, $T$  satisfies Weyl’s theorem. Since $\sigma_p(T^*)=\emptyset$, $T^*$ has SVEP. From \cite[corollary 2.45]{aiena2004fredholm}, we have $\sigma(T)=\sigma_a(T)$. Now by \cite[Proposition 1.2]{dug}, we have $\sigma_w(T)=\sigma_{uw}(T)$. Thus, \begin{center}
     $\sigma_a(T)\setminus\sigma_{uw}(T)= \sigma(T)\setminus\sigma_{w}(T)=E_0(T)=\emptyset=E(T)$.
 \end{center}
 Hence $T$ satisfies property $(UW_E)$.
 \item 	Let $T \in \mathcal{SP}$ and $E(T) \subseteq E_{0}(T^*)$. Then from \cite[Proposition 1.2]{dug}, either
	\begin{align*}
		E_{0}(T^*) = \emptyset  ~\text{or}~ E_{0}(T^*)=\lbrace \alpha \rbrace
	\end{align*}
	  for some nonzero $\alpha \in \mathbb{C} \setminus \sigma_b(T)$. Hence we obtain that there is an $\alpha \in \mathbb{C} \setminus \sigma_b(T)$ such that $E(T) \subseteq \lbrace \alpha \rbrace$.\par Conversely, suppose that $T \in \mathcal{SP}, E(T) \subseteq E_{0}(T^*)$, and there is an $\alpha \in \mathbb{C} \setminus \sigma_b(T)$ such that $E(T) \subseteq \lbrace \alpha \rbrace$. Since $T$ is a supercyclic operator, $T^*$ has the SVEP and hence $\sigma(T)= \sigma_a(T)$. From \cite[Proposition 1.2]{dug}, we have  $$\sigma_w(T)=\sigma_{uw}(T).$$ Thus,
	\begin{equation*}
		\begin{split} 
			\sigma_a(T) \setminus \sigma_{uw}(T)&=\sigma(T) \setminus \sigma_w(T)\subseteq P_{00}(T) \subseteq \pi_{00}(T) \subseteq \pi_0(T) \subseteq \pi_{00}(T^*)\\&= P_{00}(T^*)=P_{00}(T).
		\end{split}
	\end{equation*} Therefore, $T$ satisfies property $(UW_E)$. 
        \end{enumerate}
\end{proof}
\begin{corollary}
    If $T \in \mathcal{HP}\cup\mathcal{SP}$ and $\bigcap_{n=1}^{\infty}\mathcal{R}(T^n)=\lbrace0\rbrace$, then $T$ satisfies property $(V_E)$, and consequently property $(UW_E)$.
\end{corollary}
\begin{proof}
	Since $\bigcap_{n=1}^{\infty}\mathcal{R}(T^n)=\lbrace0\rbrace$, it follows from \cite[Proposition 2]{schmoeger1994operators} that $\sigma(T)$ is connected. Then by Theorem \ref{th5}, the result follows.
\end{proof}

It is also true that an operator $T \in \mathcal{HP}$ satisfies property $(V_E)$ if and only if $\pi_0(T)=\emptyset$. Denote $\sigma_{le}(T)=\lbrace \lambda \in \mathbb{C} : T- \lambda I \notin U_{SF}(\mathcal{H}) \rbrace$, $\sigma_{re}(T)=\lbrace \lambda \in \mathbb{C} : T- \lambda I \notin L_{SF}(\mathcal{H}) \rbrace$, and $\partial \sigma$ is the boundary of the set $\sigma$.

\begin{corollary}
Let $T \in \mathcal{HP}$ with any one of the following conditions:
		\begin{enumerate}
		\item $\sigma_e(T) \cap \partial \mathbb{D} =\emptyset$.
	\item $\sigma_{le}(T) \cap \partial \mathbb{D} =\emptyset$.
	\item $\sigma_{re}(T) \cap \partial \mathbb{D} =\emptyset$.
	\item $\partial \sigma_e(T) \cap \partial \mathbb{D} =\emptyset$.
\end{enumerate}
	Then $T$ satisfies property $(V_E)$ and consequently property $(UW_E)$. 
\end{corollary}
\begin{proof}
	It follows from \cite[Corollary 1.4]{dug} that if we assume any one of the conditions (1)-(4), we get that $\sigma(T)$ is connected. Then the result follows  by Theorem \ref{th5}.
\end{proof}
\begin{remark}
     Let $L^\infty(\mathbb{T})$ be the space of all essentially bounded functions in the unit circle $\mathbb{T}, L^2(\mathbb{T})$ be the space of all square-integrable functions and $H^2(\mathbb{T})$ be the Hardy space consisting of all analytic functions on $\mathbb{T}$. For $\phi \in L^\infty(\mathbb{T})$, the Toeplitz operator on $H^2(\mathbb{T})$ is given by
	 \begin{center}
	 	$T_\varPhi(f)= P(\varPhi f)$,
	 \end{center}
	 where $f \in H^2(\mathbb{T})$ and $P$ is the orthogonal projection of $L^2(\mathbb{T})$ onto $H^2(\mathbb{T})$. If $T_\phi \in \mathcal{HP}\cup\mathcal{SP}$ and $\phi \in L^\infty(\mathbb{T})$ are non-constant, then $T_\phi$ on the Hardy space satisfies property $(V_E)$, property $(UW_E)$ and property $(W_E)$. This follows from Theorem \ref{th5} due to the fact that $\sigma(T_\phi)$ is connected \cite{wang}.
\end{remark}
Denote by $\overline{\mathcal{HP}}$ the norm closure of the class of hypercyclic operators, and by $\overline{\mathcal{SP}}$ that of supercyclic operators. The following example illustrates that operators can satisfy property $(UW_E)$ without belonging to the class $\overline{\mathcal{HP}}$.
\begin{example}

	\normalfont
	Let $A \in B(l^2 \oplus l^2)$ be defined as 
	\begin{equation*}
		A = 
		\begin{bmatrix}
			T_1 &  0\\
			0 & T_2 
		\end{bmatrix}, 
	\end{equation*}	where 
	\begin{align*}
		T_1(x_1,x_2,x_3,\ldots)=\left(x_1,\frac{x_2}{2},\frac{x_3}{3},\ldots\right), 
		\end{align*}
	\begin{align*} T_2(x_1,x_2,x_3,\ldots)=\left(0,0,\frac{x_2}{2},\frac{x_3}{3},\ldots\right).
	\end{align*} Then
	\begin{equation*}
		\sigma(A)=\sigma_a(A)= \lbrace 0,1,\frac{1}{2}, \frac{1}{3},\ldots\rbrace, ~\sigma_{uw}(A)=\sigma_b(A)=\lbrace0\rbrace 
	\end{equation*}and 
\begin{equation*}
	\pi_0(A)=\lbrace 1,\frac{1}{2}, \frac{1}{3},\ldots\rbrace.
	\end{equation*} 
Thus, $A$ satisfies property $(UW_E)$. Since $\sigma_w(A) \cup \partial \mathbb{D}$ is not connected and $\sigma(A) \neq \sigma_b(A)$, $A \notin \overline{\mathcal{HP}}$ by \cite[ Theorem 2.1]{herr}.
\end{example}
The following example shows that there are operators in $\overline{\mathcal{HP}}$ which do not satisfy property $(V_E)$.
\begin{example}

	\normalfont
	Let $V \in B(l^2 \oplus l^2)$ be given as
	\begin{equation*}
		V = 
		\begin{bmatrix}
			A &  0\\
			0 & B 
		\end{bmatrix}, 
	\end{equation*}where
	\begin{align*}
		A(x_1,x_2,x_3,\ldots)=(0,x_1,x_2,x_3,\ldots),B(x_1,x_2,x_3,\ldots)=(x_2,x_3,\ldots).
	\end{align*}
	It can be seen that $\pi_0(V)=\emptyset$, $0 \in \sigma(V)$ and also $0 \notin \sigma_{uw}(V)(= \partial\mathbb{D})$. Thus, $0 \in \sigma(V) \setminus \sigma_{uw}(V)$ but $ 0\notin \pi_0(V)$. Therefore, $V$ does not satisfy property $(V_E)$. From \cite[Theorem 2.1]{herr}, it follows that $V \in \overline{\mathcal{HP}}$.
	
\end{example}
There are operators that do not belong to $\overline{\mathcal{HP}}$ and do not satisfy property $(V_E)$ as we see in the following example. 
\begin{example}
	\normalfont
	Let $T \in B(l^2)$ be defined as 
	\begin{align*}
		T(x_1,x_2,x_3,x_4,x_5,\ldots)=\left(0,x_1,0,\frac{x_3}{3},0,\frac{x_5}{5},\ldots\right). 
	\end{align*}
	Then \begin{center}$\sigma(T)=\sigma_a(T)=\sigma_{uw}(T)=\sigma_{w}(T)=\pi_0(T)=\lbrace 0 \rbrace$.\end{center} This shows that $\sigma_w(T) \cup \partial \mathbb{D}$ is not connected and hence $T \notin \overline{\mathcal{HP}}$ by \cite[ Theorem 2.1]{herr}. Also, $T$ does not satisfy property $(UW_E)$.
\end{example}
The following is an example of operator $T$ that satisfies property $(UW_E)$ but not in $\overline{\mathcal{SP}}$.
\begin{example}
	\normalfont
	Let $T$ be an operator on $l^2$, given by 
	\begin{center}
		$T(x_1,x_2,x_3,\ldots)=(0,x_1,0,x_2,0,x_3,\ldots)$.
	\end{center}
	 We obtain that
	\begin{center}
		$\sigma(T)= \overline{\mathbb{D}}$, $\sigma_a(T)=\partial\mathbb{D}=\sigma_{uw}(T)$ and $\pi_0(T)=\emptyset$.
	\end{center}
	  Thus, $T$ satisfies property $(UW_E)$. Since 
	\begin{center}
		 $\lbrace \lambda \in \rho_{SF}(T): ind(T-\lambda I)<0 \rbrace \neq \emptyset,$
	\end{center}  $T \notin \overline{\mathcal{SP}}$ by \cite[Theorem 3.3]{herr}. 
\end{example} 
The above example is discussed in \cite{liu} to show that $a$-Weyl’s theorem is not a sufficient condition for $T \in \overline{\mathcal{HP}}$. Bakkali and Tajmouati \cite{elb} studied operators in $\overline{\mathcal{HP}}$ and $\overline{\mathcal{SP}}$ that satisfy property $(w)$ and characterized them. The following theorem is a characterization of operators in the norm closure of hypercyclic and supercyclic operators that satisfy property $(UW_E)$. 
\begin{theorem}
	Suppose that $T \in \mathcal {B(H)}$ has property $(UW_E)$. Then the following statements hold.
	\begin{enumerate}
	\item $T \in \overline{\mathcal{HP}}$ if and only if $\sigma(T) \cup \partial \mathbb{D}$ is connected.
	\item  $T \in \overline{\mathcal{SP}}$ if and only if $\sigma(T) \cup \partial (r{\mathbb{D}})$ is connected for some $r \geq0$, where $\partial (r{\mathbb{D}})$ is the boundary of the disc $r\mathbb{D}=\lbrace  \lambda \in \mathbb{C}: | \lambda| < r \rbrace$.
\end{enumerate}
\end{theorem}
\begin{proof}
	
	\item(1) Let $T \in \overline{\mathcal{HP}}$. Then by \cite[Theorem 2.1]{herr}, $\sigma_{w}(T) \cup \partial \mathbb{D}$ is connected, $\sigma(T)\setminus \sigma_b(T)=\emptyset$ and $ind(T-\lambda I) \geq 0$ for all $\lambda \in \rho_{SF}(T)$. Since $T$ satisfies property $(UW_E)$, $T$ satisfies Browder's theorem. Therefore, $\sigma_b(T)=\sigma_{w}(T)$ and hence $\sigma_w(T)=\sigma(T)$. Thus, $\sigma(T) \cup \partial \mathbb{D}$ is connected. A similar argument as in \cite[Theorem 2.1]{elb} gives the converse part.
	\item(2) Follows from \cite[Theorem 3.3]{herr} and \cite[Theorem 2.1]{elb}.
\end{proof}
\section{Weyl type theorems and Toeplitz operators on the Bergman Space}
The invariance of various Weyl type theorems under compact perturbations has been studied by numerous researchers \cite{aiena,jia,prasad,yang}. In particular, Jia and Feng \cite{jia} studied the stability of variants of Weyl type theorems under compact perturbations in connection with the connectedness of the complements of the Weyl and upper semi-Weyl spectra. \par
Now we study the invariance of property $(UW_E)$ under compact perturbations. It follows from \cite[Corollary 3.4]{jia} that if $T \in \mathcal{B(H)}$
and $\lambda \in \text{iso}~
\sigma_w(T)$, then there exists a compact operator $K \in \mathcal{K}(\mathcal{H})$
such that $\lambda \in E(T+K)$. We prove
that property $(UW_E)$ is invariant under compact perturbation if and only
if the set of isolated points of the Weyl spectrum is empty and the complement of the upper semi-Weyl spectrum is connected.
\begin{theorem} \label{pra}
	Let $T\in \mathcal {B(H)}$. Then $T+K$ satisfies property $ (UW_{E})$ for all $K \in \mathcal {K(H)}$ if and only if\\ \rm{(1)} $\rm{iso}$~$\sigma_{w}(T)= \emptyset$\\ \rm{(2)} $\sigma_{uw}(T)$ is simply connected.
\end{theorem}
\begin{proof}
    Suppose that $ \text{iso}~ \sigma_{w}(T) =\emptyset$ and $\mathbb{C} \setminus \sigma_{uw}(T)$ is connected. To prove $T+K$ satisfies property $(UW_{E}) $ for all $K \in \mathcal {K(H)}$. Using \cite[Lemma 2.2]{prasad}, it is enough to prove the following conditions: $E(T+K) \subseteq \rho_{SF}^-(T+K) \cup \rho_{SF}^0(T+K)$ and $[\rho_{SF}^-(T+K) \cup \rho_{SF}^0(T+K)]\cap \sigma_{p}(T+K)$. 
    \par To prove $E(T+K) \subseteq \rho_{SF}^-(T+K) \cup \rho_{SF}^0(T+K)$, choose a $\lambda \in E(T+K)$. Then $$\lambda \in \text{iso}~\sigma (T+K) \cap \sigma_p(T+K).$$ This shows that there exist a $\delta >0$ such that $T+K-\mu$ is invertible for $\mu \in B(\lambda, \delta)\setminus \lbrace \lambda \rbrace$. Hence $$B(\lambda, \delta)\setminus \lbrace \lambda \rbrace \subseteq \rho_{SF}^0(T+K).$$ If $\lambda \notin \rho_{SF}^-(T+K) \cup \rho_{SF}^0(T+K)$, then $\lambda \in \sigma_{lre}(T+K)$ or $\lambda \in \rho_{SF}^+(T+K)$. But $\rho_{SF}^+(T+K)$ is an open subset of $\mathbb{C}$ and is contained in the interior of $\sigma(T+K)$, a contradiction. Thus we have $\lambda \in \sigma_{lre}(T+K)$ and hence $\lambda \in \sigma_{w}(T+K)$. In addition, $$B(\lambda, \delta) \cap \sigma_{w}(T+K)=\lbrace \lambda \rbrace.$$ This implies that $\lambda \in \text{iso}~\sigma_{w}(T+K)$, which is a contradiction to our assumption. Therefore, $$E(T+K) \subseteq \rho_{SF}^-(T+K) \cup \rho_{SF}^0(T+K).$$
To prove the second condition, we know that $\sigma_{uw}(T+K)=\sigma_{uw}(T)$.
	Since $\mathbb{C} \setminus \sigma_{uw}(T)$ is connected, it follows that $[\mathbb{C} \setminus \sigma_{uw}(T+K)] \cap \sigma_p(T+K)$ consists of only normal eigenvalues because $\rho_{SF}^0(T+K)$ is not empty. Therefore, \begin{center}
	    $[\rho_{SF}^-(T+K) \cup \rho_{SF}^0(T+K)]\cap \sigma_{p}(T+K)$ is a discrete set.
        \end{center}
	Since $K$ is arbitrary, we conclude that $T+K$ satisfies property $(UW_E)$ for all $K \in \mathcal {K(H)}$.
    \par Conversely, assume that $T+K$ satisfies property $(UW_E)$ for all $K \in \mathcal {K(H)}$. We have to prove that $ \text{iso}~ \sigma_{w}(T) =\emptyset$ and $\mathbb{C} \setminus \sigma_{uw}(T)$ is connected.\\
    If $ \text{iso}~ \sigma_{w}(T) \neq \emptyset$, then there exists $K \in \mathcal{K(H)}$ such that $\lambda \in E(T+K)$, where $\lambda \in  \text{iso}~ \sigma_{lre}(T)$. Since $$\lambda \in \sigma_{w}(T)=\sigma_{w}(T+K),~ \lambda \notin \rho_{SF}(T+K).$$ Therefore, $\lambda \notin \rho_{SF}^-(T+K) \cup \rho_{SF}^0(T+K)$. Hence by \cite[Lemma 2.2]{prasad}, $T+K$ does not satisfy property $(UW_E)$.\\
	If $\mathbb{C} \setminus \sigma_{uw}(T)$ is not connected, that is $\rho_{SF}^-(T) \cup \rho_{SF}^0(T)$ is not connected, then there is a bounded component $\Delta$ of $\rho_{SF}^-(T) \cup \rho_{SF}^0(T)$. By \cite[Proposition 3.2]{jia}, there exist a compact operator $K$ such that the bounded component is contained in the point spectrum of $T+K$. Since $\Delta \subseteq \rho_{SF}^-(T+K)\cup \rho_{SF}^0(T+K)$, $T+K$  does not satisfy property $(UW_E)$ because $\Delta \subseteq [\rho_{SF}^-(T+K)\cup \rho_{SF}^0(T+K)] \cap \sigma_{p}(T+K)$ by \cite[Lemma 2.2]{prasad}.
    \end{proof}
    Coburn \cite{cob} proved that Toeplitz operators on the Hardy space satisfy Weyl's theorem. Extensive studies of Weyl type theorems for Toeplitz operators on the Hardy space can be found in \cite{cob,jia}. The invertibility of Toeplitz operators on the Bergman space with harmonic polynomial symbols was studied by Guan and Zhao \cite{guan}. Guo, Zhao and Zheng \cite{guo} studied the topological behavior of Toeplitz operators on the Bergman space with harmonic polynomial symbols and proved that if $q$ is in the disc algebra $H^\infty \cap C(\overline{\mathbb{D}})$ and $h(z)=\overline{z}+q(z)$,
then $T_h$ satisfies Weyl's theorem,where $H^{\infty}$ is the algebra of all bounded analytic functions on $\mathbb{D}$. They also found an example of a non-hyponormal Toeplitz operator on
the Bergman space that satisfies Weyl's theorem.
\par Let $\mathbb{D} \subset \mathbb{C}$ be the open unit disk and $dA(z)$ be the Lebesgue area measure on $\mathbb{D}$, which is normalized so that the measure of $\mathbb{D}$ is one. $L^\infty(\mathbb{D})$ is the space of bounded measurable functions on $\mathbb{D}$.
The \textit{Bergman space} $A^2(\mathbb{D})$ is the closed subspace of $L^2(\mathbb{D})$ consisting of functions that are analytic on $\mathbb{D}$, where $L^2(\mathbb{D})$ is the space of all square integrable functions on $\mathbb{D}$. Let ${P}$ from  $L^2(\mathbb{D})$ onto $A^2(\mathbb{D})$ be an orthogonal projection. For $\phi \in L^\infty (\mathbb{D})$, the Toeplitz operator $T_\phi$ on the Bergman space is given by
\begin{center}
	$T_\phi(f)= {P}( \phi f)$
\end{center} for $f$ in the Bergman space $A^2(\mathbb{D})$.\\The \textit{Harmonic Bergman space}, $L^2_h(\mathbb{D})$, is the closed subspace of $L^2(\mathbb{D})$ consisting of all complex-valued harmonic functions on $\mathbb{D}$. For $\phi \in L^\infty(\mathbb{D})$, the Toeplitz operator $\tilde{T_\phi}$ on the harmonic Bergman space is given by
\begin{center}
	$\tilde{T_\phi}(f)= Q(\phi f )$,
\end{center} where $Q$ is the orthogonal projection from $L^2(\mathbb{D})$ to $L^2_h(\mathbb{D})$ and $f \in L^2_h(\mathbb{D})$.

\begin{lemma}\cite{stroe,sundberg}\label{su}
Suppose that $\phi \in C(\overline{\mathbb{D}})$. Then the essential spectrum of $T_\phi$ is given by $\sigma_e(T_\phi) = \phi(\partial \mathbb{D})$.
		Moreover, if $T_\phi$ is a Fredholm operator, then the Fredholm index of $T_\phi$ is given by
		\begin{center}
			$ind(T_\phi) = \alpha(T_\phi) - \beta(T_\phi) = -\operatorname{wind} (\phi(\partial \mathbb{D}), 0)$,
		\end{center}
		where $\operatorname{wind} (\phi(\partial \mathbb{D}), 0)$ is the winding number of the closed oriented curve $\phi(\partial \mathbb{D})$ with respect to the origin, which is defined by
		\begin{align*}
			\operatorname{wind} (\phi(\partial \mathbb{D}), 0)= \frac{1}{2 \pi i} \int_{\phi(\partial \mathbb{D})} \frac{dz}{z}.
		\end{align*}
        \end{lemma}
        An interesting observation from \cite[Proposition 8]{axler} is that the self-commutator $[T_\phi, T_g]=T_\phi T_g - T_gT_\phi$ is a compact operator for functions $\phi,g \in C(\overline{\mathbb{D}})$.
	Since  $\phi \in C(\overline{\mathbb{D}})$, we have $\overline{\phi} \in C(\overline{\mathbb{D}})$ and $[T_\phi, T_{\overline{\phi}}]=T_\phi T_{\overline{\phi}} - T_{\overline{\phi}}T_\phi$ is compact.
  Thus, $T_\phi$ is essentially normal because $T_\phi^*=T_{\overline{\phi}}$. Then  we have, $\sigma_{e}(T_{\phi})=\sigma_{lre}(T_{\phi}).$  Therefore, by Lemma \ref{su},
	\begin{center}
		$\sigma_{e}(T_{\phi})=\sigma_{lre}(T_{\phi})=\phi(\partial \mathbb{D})$.
	\end{center}
    Then the Weyl spectrum and upper semi-Weyl spectrum of $T_\phi$ can be written as,
\begin{center}
	$\sigma_{w}(T_\phi)=\sigma_{lre}(T_\phi) \cup \rho_{SF}^+(T_\phi) \cup \rho_{SF}^-(T_\phi) 
	\vspace{.2cm}
	= \phi(\partial \mathbb{D}) \cup \lbrace \lambda \in \mathbb{C}: \operatorname{wind} (\phi(\partial \mathbb{D}), \lambda)<0 \rbrace \cup \lbrace \lambda \in \mathbb{C}: \operatorname{wind} (\phi(\partial \mathbb{D}), \lambda)>0 \rbrace.$
		\end{center}
	
and
\begin{equation*}
\begin{split}	\sigma_{wa}(T_\phi)&= \sigma_{lre}(T_\phi) \cup \rho_{SF}^+(T_\phi)\\&=\phi(\partial \mathbb{D}) \cup \lbrace \lambda \in \mathbb{C}: \operatorname{wind} (\phi(\partial \mathbb{D}), \lambda)<0 \rbrace.
\end{split}
	\end{equation*}
    \begin{theorem}\label{uw}
Let $\phi \in C(\overline{\mathbb{D}})$. Then
\begin{enumerate}
\item   $T_\phi + K$ satisfies property $ (UW_E)$ for all $K \in \mathcal{K}(A^2(\mathbb{D}))$ if and only if $\phi$ is not constant on $\partial \mathbb{D}$ and the winding number of $\phi$ with respect to each hole of $\phi(\partial \mathbb{D})$ is negative.
\item  $ T_\phi + K$ satisfies Weyl's theorem for all $K \in \mathcal{K}(A^2(\mathbb{D}))$ if and
only if $\phi$ is not constant on $\partial \mathbb{D}$ and the winding number of $\phi$ with respect to each hole of $\phi(\partial \mathbb{D})$ is nonzero.\\
$T_\phi + K$ satisfies $ a$-Weyl's theorem for all $K \in \mathcal{K}(A^2(\mathbb{D}))$ if and
only if $\phi$ is not constant on $\partial \mathbb{D}$ and the winding number of $\phi$ with respect to each hole of $\phi(\partial \mathbb{D})$ is negative.\\
$T_\phi + K$ satisfies Browder's theorem for all $K \in \mathcal{K}(A^2(\mathbb{D}))$ if and
only if the winding number of $\phi$ with respect to each hole of $\phi(\partial \mathbb{D})$ is nonzero.\\
$T_\phi + K$ satisfies $a$-Browder's theorem for all $K \in \mathcal{K}(A^2(\mathbb{D}))$ if and
only if the winding number of $\phi$ with respect to each hole of $\phi(\partial \mathbb{D})$ is negative.
\end{enumerate}
\end{theorem}
\begin{proof}
\begin{enumerate}
\item We know that $\sigma_{uw}(T_\phi)$ consists of $\phi(\partial \mathbb{D})$ and those holes with respect to which the winding number of $\phi$ is negative. Then $\mathbb{C} \setminus \sigma_{uw}(T_\phi)$ is connected if and only if the winding
number of $\phi$ with respect to each hole of $\phi(\partial \mathbb{D})$ is negative.
We claim: 
\begin{center}
$\text{iso}~\sigma_{w}(T_\phi)=\emptyset$ if and only if $\phi$ is non-constant on $\partial \mathbb{D}$.
\end{center}  Let $\text{iso}~\sigma_{w}(T_\phi)\neq \emptyset$. Then by \cite[Lemma 2.12]{jia} it follows that $\text{iso}~\sigma_{lre}(T_\phi)\neq \emptyset$. Since 
\begin{center}
 $\sigma_{e}(T_{\phi})=\sigma_{lre}(T_{\phi })=\phi(\partial \mathbb{D})$ is connected, $\phi$ is constant on $\partial \mathbb{D}$.
\end{center}
If $\phi \equiv \lambda$, then $\sigma(T_\phi)=\sigma_{w}(T_\phi)=\lbrace \lambda \rbrace$. Hence, if $\phi$ is constant on $\partial\mathbb{D}$, then $\text{iso}~ \sigma_{w}(T_\phi) \neq \emptyset$. Then from Theorem \ref{pra}, the result follows. 
\item It can be seen in \cite[Theorem 1.2]{jia} that $ T_\phi + K$ satisfies Weyl's theorem for all $K \in \mathcal {K(H)}$ if
\begin{center}
    $\text{iso}~ \sigma_w(T_\phi)=\emptyset$ and $\mathbb{C} \setminus \sigma_{w}(T_\phi)$ is connected.
    \end{center}We have $\sigma_{w}(T_\phi)$ consists of $\phi(\partial \mathbb{D})$ and those holes with respect to which the winding number of $\phi$ is nonzero. Thus, $\mathbb{C} \setminus \sigma_{w}(T_\phi)$ is connected if and only if the winding number of $\phi$ with respect to each hole of $\phi(\partial \mathbb{D})$ is nonzero. By a similar argument as in condition (a),  $\text{iso}~ \sigma_w(T_\phi)= \emptyset$ if and only if $\phi$ is not constant on $\partial \mathbb{D}$. Similarly, using the above arguments and from \cite{jia}, we can also prove the other implications.
\end{enumerate}
\end{proof} 
\begin{remark}\normalfont
 If  we take $\phi \in H^\infty \cap C(\overline{\mathbb{D}})$ and $\phi$ has no zeros in $\mathbb{D}$, then by the maximum and minimum modulus principles for analytic functions, $\phi$ is non-constant on $\mathbb{D}$. Therefore, in that case, we can modify Theorem \ref{uw} by adding the condition that $\phi$ is not constant on $\mathbb{D}$. That is, $T_\phi + K$ satisfies property $ (UW_E)$ for all $K \in \mathcal{K}(A^2(\mathbb{D}))$ if and only if $\phi$ is non-constant on $\mathbb{D}$ and the winding number of $\phi$ with respect to each hole of $\phi(\partial \mathbb{D})$ is negative.
\end{remark}
\par The following is a class of Toeplitz operators on the Bergman space that do not satisfy property $(UW_E)$.
\begin{example}\label{ex29}
	\normalfont
	Let $J>1$ and $h(re^{i\theta})=\chi_{{\frac{1}{J}}\mathbb{D}}(re^{i\theta}) e^{-i \theta}$, where $z=re^{i\theta} \in \mathbb{D}$ and $\frac{1}{J}\mathbb{D}= \lbrace z \in \mathbb{C}: |z|<\frac{1}{J}\rbrace$. Since $\lbrace e_n(z)\rbrace_{n=0}^{\infty} =\lbrace \sqrt{n+1}z^n\rbrace_{n=0}^{\infty}$ is an orthonormal basis of the Bergman space $A^2(\mathbb{D})$, we can express $T_h$ with respect to this orthonormal basis as 
	\begin{align*}
		\begin{split}
			T_h(e_n)(z) &= \sum_{k=1}^{\infty} \langle T_h(e_n(z)), e_k(z) \rangle e_k(z)\\
			&= \sum_{k=1}^{\infty} \langle h(z)e_n(z), e_k(z) \rangle e_k(z)\\
			&= \sum_{k=1}^{\infty}\left[ \int_{\mathbb{D}} h(z)e_n(z) \overline{e_k(z)} dA(z)\right] e_k(z)\\
			&=\sum_{k=1}^{\infty} \left[\frac{1}{\pi}\int_0^{2\pi}\int_0^{\frac{1}{J}}r^{n+k+1}e^{i(n-k-1)\theta} dr d\theta \right]e_k(z)\\
			&= \frac{\sqrt{n(n+1)}}{2n+1}\left(\frac{1}{J}\right)^{2n} e_{n-1}(z).
		\end{split}
	\end{align*}
	Thus, the Toeplitz operator with symbol $h$ on the space $A^2(\mathbb{D})$ is
	\begin{align*}
		T_h(e_n)(z) = 
		\begin{cases}
			0 & \text{if } n=0\\
			\frac{\sqrt{n(n+1)}}{2n+1}\left(\frac{1}{J}\right)^{2n} e_{n-1}(z) & \text{if } n\geq 1.
		\end{cases} 	
	\end{align*}
	Since 
	\begin{align*}
		\lim_{n \rightarrow \infty} \frac{\sqrt{n(n+1)}}{2n+1}\left(\frac{1}{J}\right)^{2n} = 0,
	\end{align*} $T_h$ is compact backward weighted shift operator. From \cite[Proposition 27.7]{con}, we get $$\sigma(T_h)=\sigma_a(T_h)=\sigma_{uw}(T_h)=\pi_0(T_h)=\lbrace0\rbrace.$$ Since $\sigma_a(T_h),\sigma_{uw}(T_h)$, and $\pi_0(T_h)$ have elements in common, the Bergman-Toeplitz operator $T_h$ with $h(re^{i\theta})=\chi_{{\frac{1}{J}}\mathbb{D}}(re^{i\theta}) e^{-i \theta}$, $J>1$ does not satisfy property $(UW_E)$.
\end{example}
\begin{remark}
	\normalfont
	\begin{enumerate}
		\item[(a)] From Example \ref{ex29}, it is evident that $T_h$ does not satisfy Weyl's theorem and hence $T_h$ does not satisfy spectral properties that are stronger variants of Weyl's theorem (see \cite{sana}).
		\item[(b)] If $p(z)$ is a polynomial of degree less than 3 and $T_{\overline{z}+p(z)} \in \mathcal{HP}\cup\mathcal{SP}$ on the Bergman space, then it is easy to see that $T_{\overline{z}+p(z)}$ satisfies property $(UW_E)$ by Theorem \ref{th5} and \cite[Theorem 3.1]{guo}.
	\end{enumerate}
\end{remark}
 Suppose that $q$ is in the disk algebra $H^\infty \cap C(\overline{\mathbb{D}})$ and let $h(z) = \overline{z} + q(z)$. From \cite[Theorem 5.2]{guo}, it follows that Weyl's theorem holds for the Bergman-Toeplitz operator $T_h$. The spectrum of $T_h$ is given by 
\begin{center}
	$\sigma(T_h)= h(\partial \mathbb{D}) \cup \lbrace \lambda \in \mathbb{C}: \lambda \notin h(\partial \mathbb{D}) \text{ and } \operatorname{wind}(h(\partial \mathbb{D}), \lambda) \neq 0 \rbrace \cup \Delta$,
\end{center}
	where $\Delta$ is a subset of set of eigenvalues of $T_h$ given by
	\begin{center}
		$\Delta = \sigma_p(T_h) \cap \lbrace \lambda \in \mathbb{C}: \lambda \notin \sigma_{e}( T_h) ~\text{and} ~ind(T_h- \lambda I) = 0 \rbrace$.
	\end{center}
	The Weyl spectrum and upper semi-Weyl spectrum are given by
	\begin{align*}
		\sigma_w(T_h)= h(\partial \mathbb{D}) \cup \lbrace \lambda \in \mathbb{C}: \lambda \notin h(\partial \mathbb{D}) \rm~\text{and}~ \operatorname{wind}(h(\partial \mathbb{D}), \lambda) \neq 0 \rbrace,~ \text{and}\\
		\sigma_{uw}(T_h)= h(\partial \mathbb{D}) \cup \lbrace \lambda \in \mathbb{C}: \lambda \notin h(\partial \mathbb{D}) \rm~\text{and}~ \operatorname{wind} (h(\partial \mathbb{D}), \lambda) < 0 \rbrace.
	\end{align*}
	All isolated points of $\sigma(T_h)$ are contained in $\Delta$. Thus, $\pi_0(T_h)\subseteq \Delta$ and $\Delta \subseteq \pi_{00}(T_h) \subseteq \pi_0(T_h)$ and so $\pi_0(T_h)= \Delta$ (see \cite [Theorem 5.2] {guo}).
	Then the spectrum of $T_h$ can be expressed as
	\begin{center}
		$\sigma(T_h)=\sigma_{uw}(T_h) \cup \lbrace \lambda \in \mathbb{C}: \lambda \notin h(\partial \mathbb{D}) \rm~\text{and} \rm~\operatorname{wind} $$(h(\partial \mathbb{D}), \lambda) > 0 \rbrace \cup \pi_0(T_h) $.
	\end{center}
	\par Some examples of Toeplitz operators that satisfy property $(UW_E)$ on the Bergman space is given as follows.
	\begin{example}\label{cho}
		\normalfont
		If we take $q(z)$ in the above note as zero function, then $h(z)= \overline{z}$. Hence the operator $T_h$ is given by 
		\begin{align*}
			T_h(e_n)(z) = 
			\begin{cases}
				0 & \text{if } n=0\\
				\sqrt{\frac{n}{n+1}} e_{n-1}(z) & \text{if } n\geq 1
			\end{cases}, 
		\end{align*}
		where $\lbrace e_n(z)\rbrace_{n=0}^{\infty} =\lbrace \sqrt{n+1}z^n\rbrace_{n=0}^{\infty}$ is an orthonormal basis of the Bergman space. Since $h(\partial \mathbb{D})$ has orientation in the clockwise direction, we have 
		\begin{center}
			$	\lbrace \lambda \in \mathbb{C}: \lambda \notin h(\partial \mathbb{D}) \rm~\text{and} \rm~\operatorname{wind} (h(\partial \mathbb{D}), \lambda) > 0 \rbrace=\emptyset$.
		\end{center}
		This shows that 
		\begin{center}
			$\sigma(T_h)=\sigma_{uw}(T_h) \cup \pi_0(T_h) $.
		\end{center}
		Thus, $T_h$ satisfies property $(V_E)$ and therefore $T_h$ satisfies the property ($UW_E$). In addition, $\rm{iso}~$$\sigma_{w}(T_h)=\rm{iso}~$$\sigma_{uw}(T_h)=\emptyset$. By \cite[Lemma 2.3]{paiena}, $T_h$ is polaroid and $a$-isoloid.
	\end{example}
	
	\begin{example}
		\normalfont
		Let $q(z)= \frac{z^2}{3}$. Then $h(z)= \overline{z}+ \frac{z^2}{3}$. If we take $f(z)= \frac{z^2}{3}$ and $g(z)=z$, we get $f'(1)=\frac{2}{3}$ and $g'(1)=1$. Thus, $\left| f'(1)\right| < \left| g'(1)\right|$. We know that if $T_{f+\overline{g}}$ is hyponormal, then $\left| f'(z)\right| \geq \left| g'(z)\right|$ for all $z \in \partial \mathbb{D}$ \cite{guo}. Then $T_h$ is not hyponormal. Since $h(\partial \mathbb{D})$ trace the curve in the clockwise direction, we have 
		\begin{center}
			$	\lbrace \lambda \in \mathbb{C}: \lambda \notin h(\partial \mathbb{D}) \rm~\text{and} \rm~\operatorname{wind} (h(\partial \mathbb{D}), \lambda) > 0 \rbrace=\emptyset$.
		\end{center}
		Hence
		\begin{center}
			$\sigma(T_h)=\sigma_{uw}(T_h) \cup \pi_0(T_h) $.
		\end{center} Thus, $T_h$ satisfies property $(V_E)$. That is, there exist a non-hyponormal operator on the Bergman space that satisfies property $(V_E)$ and so does property $(UW_E)$. We have $ h(\partial \mathbb{D}) \cup \lbrace \lambda \in \mathbb{C}: \lambda \notin h(\partial \mathbb{D}) \rm~\text{and} \rm~\operatorname{wind} (h(\partial \mathbb{D}), \lambda) \neq 0 \rbrace$ and $h(\partial \mathbb{D}) \cup \lbrace \lambda \in \mathbb{C}: \lambda \notin h(\partial \mathbb{D}) \rm~\text{and} \rm~\operatorname{wind} (h(\partial \mathbb{D}), \lambda) < 0 \rbrace$ are connected. Thus, $\rm{iso}~$$\sigma_{w}(T_h)=\rm{iso}~$$\sigma_{uw}(T_h)=\emptyset$. Then, $T_h$ is polaroid and $a$-isoloid.
	\end{example}
	In general $T_h$, where $h(z)= \overline{z}+\frac{z^{n-1}}{n}$ and $n$ is a positive integer satisfies property $(UW_E)$ and $(V_E)$.
	If we take $q(z)\in   H^ \infty \cap C(\overline{\mathbb{D}})$ and $h(z)=\overline{z}+q(z)$, we have  
	\begin{center}
		$	\lbrace \lambda \in \mathbb{C}: \lambda \notin h(\partial \mathbb{D}) \rm~\text{and} \rm~\operatorname{wind} (h(\partial \mathbb{D}), \lambda) < 0 \rbrace \neq \emptyset$.
	\end{center}
	Then by \cite[Corollary 2.106]{aiena}, $T_h$ cannot have SVEP. Consequently, $T_h$ is not a class $A$ operator by \cite[Lemma 2.5]{mec}. 
	\par 	The following example demonstrates that property $(UW_E)$ can be satisfied by a non-hyponormal operator on the harmonic Bergman space.
	\begin{example}
		\normalfont
		Let $p(z)=z^n, n\geq 1$.
		Since $z^n \in C(\overline{\mathbb{D}})$ and from \cite[Lemma 2.1]{wang}, $$\sigma_e(\tilde{T_p})=p(\partial \mathbb{D})=\lbrace \lambda \in \mathbb{C}: |z|=1\rbrace.$$ We know that
		$\sigma_{w}(\tilde{T_p})=p(\partial \mathbb{D}) \cup \lbrace \lambda \in \mathbb{C}: \lambda \notin p(\partial \mathbb{D})~ \text{and}~ind(\tilde{T_p}) \neq 0 \rbrace$. Since $\tilde{T_p}$ is Fredholm and again by \cite[Lemma 2.1]{wang}, $ind(\tilde{T_p})=0$. Thus, $\sigma_{w}(\tilde{T_p})=\sigma_{e}(\tilde{T_p})=p(\partial \mathbb{D}).$ Thus, $\lbrace \lambda \in \mathbb{C}: \lambda \notin p(\partial \mathbb{D})~\text{and}~ ind(\tilde{T_p})\neq 0 \rbrace= \emptyset$. Therefore, $\sigma_{uw}(\tilde{T_p})=\sigma_{w}(\tilde{T_p})=p(\partial \mathbb{D})$ and
		
		\begin{center}
			$\sigma(\tilde{T_p})=p(\partial \mathbb{D}) \cup \lbrace \lambda \in \mathbb{C}: \lambda \notin p(\partial \mathbb{D}) \rm~\text{and}~ $$\lambda \in \sigma_p(\tilde{T_p}) \rbrace$.
				\end{center} By \cite[Theorem 3.4]{wang}, we have $\lbrace \lambda \in \mathbb{C}: \lambda \notin p(\partial \mathbb{D}) \rm~\text{and}$$ ~\lambda \in \sigma_p(\tilde{T_p}) \rbrace=\emptyset$. Thus, $\tilde{T_p}$ satisfies property $(V_E)$. Consequently,  $\tilde{T_p}$ satisfies property $(UW_E)$. It follows from \cite[Example 3.5]{wang} that $\tilde{T_p}$ is not hyponormal.
	\end{example}
	 \textbf{Conflict of interest:} The authors state no conflict of interest.\\

\textbf{Data availability:} No data was used for the research described in this paper.\\

\textbf{Acknowledgement:} The research of first author is supported by senior research fellowship of university grants commission, India.


\begin{thebibliography}{99}
\bibitem{aiena} Aiena, Pietro. \textit{Fredholm and Local Spectral Theory II: With Application to Weyl type Theorems}. Vol. 2235. Springer, 2018.

\bibitem{aiena2004fredholm}
Aiena, Pietro.
\textit{ {F}redholm and Local Spectral Theory, with Applications to
Multipliers}. Springer Science \& Business Media, 2004.

\bibitem{paiena} Aiena, Pietro, and Triolo, Salvatore. \textit{"Some remarks on the spectral properties of Toeplitz operators."} \textit{Mediterranean Journal of Mathematics} 16, no. 6 (2019): 135.

\bibitem{axler} Axler, Sheldon, Conway, John B., and McDonald, Gerard. \textit{"Toeplitz operators on Bergman spaces."} \textit{Canadian Journal of Mathematics} 34, no. 2 (1982): 466–483.

\bibitem{berb} Berberian, S. K. \textit{"An extension of Weyl's theorem to a class of not necessarily normal operators."} \textit{Michigan Mathematical Journal} 16, no. 3 (1969): 273–279.

\bibitem{berk} Berkani, Mohammed, and Kachad, Mohammed. \textit{"New Browder and Weyl type theorems."} \textit{Bulletin of the Korean Mathematical Society} 52, no. 2 (2015): 439–452.

\bibitem{cob} Coburn, Lewis A. \textit{"Weyl's theorem for non-normal operators."} \textit{Michigan Mathematical Journal} 13, no. 3 (1966): 285–288.

\bibitem{con} Conway, John B. \textit{A Course in Operator Theory}. American Mathematical Society, 2000.

\bibitem{dug} Duggal, B. P. \textit{"Weyl's theorem and hypercyclic/supercyclic operators."} \textit{Journal of Mathematical Analysis and Applications} 335, no. 2 (2007): 990–995.

\bibitem{elb} El Bakkali, Abdeslam, and Tajmouati, Abdelaziz. \textit{"Property $(w)$ and hypercyclic/supercyclic operators."} \textit{International Journal of Contemporary Mathematical Sciences} 7, no. 26 (2012): 1259–1268.

\bibitem{guan} Guan, Nanxing, and Zhao, Xianfeng. \textit{"Invertibility of Bergman Toeplitz operators with harmonic polynomial symbols."} \textit{Science China Mathematics} 63, no. 5 (2020): 965–978.

\bibitem{guo} Guo, Kunyu, Zhao, Xianfeng, and Zheng, Dechao. \textit{"The spectral picture of Bergman Toeplitz operators with harmonic polynomial symbols."} \textit{Arkiv För Matematik} 61(2) (2023): 343-374.

\bibitem{herr} Herrero, Domingo A. \textit{"Limits of hypercyclic and supercyclic operators."} \textit{Journal of Functional Analysis} 99, no. 1 (1991): 179–190.

\bibitem{hild} Hilden, H. M., and Wallen, L. J. \textit{"Some cyclic and non-cyclic vectors of certain operators."} \textit{Indiana University Mathematics Journal} 23, no. 7 (1974): 557–565.

\bibitem{jia} Jia, Boting, and Feng, Youling. \textit{"Weyl type theorems under compact perturbations."} \textit{Mediterranean Journal of Mathematics} 15 (2018): 1–13.

\bibitem{kita} Kitai, Carol. \textit{"Invariant closed sets for linear operators."} (1984): 1148–1148.

\bibitem{liu} Liu, Ying, and Cao, Xiaohong. \textit{"a-Weyl’s theorem and hypercyclicity."} \textit{Monatshefte für Mathematik} 204, no. 1 (2024): 107–125.

\bibitem{mec} Mecheri, Salah. \textit{"Weyl's theorem for algebraically class $A$ operators."} \textit{Bulletin of the Belgian Mathematical Society–Simon Stevin} 14, no. 2 (2007): 239–246.

\bibitem{prasad} Prasad, T., Thomas, Simi, and Fernandez, Shery. \textit{"Stability of property (UWE) under compact perturbations."} \textit{Rendiconti del Circolo Matematico di Palermo Series 2} 72, no. 7 (2023): 3659–3664.

\bibitem{rako} Rakocevic, V. \textit{"Operators obeying a-Weyl's theorem."} \textit{Revue Roumaine de Mathématiques Pures et Appliquées} 10 (1989): 915–919.

\bibitem{qiu} Qiu, Sinan, and Cao, Xiaohong. \textit{"Property $(UW_E)$ for operators and operator matrices."} \textit{Journal of Mathematical Analysis and Applications} 509, no. 2 (2022): 125951.

\bibitem{sana} Sanabria, Jose, Vasquez, Luis, Carpintero, Carlos, Rosas, Ennis, and Garcia, Orlando. \textit{"On strong variations of Weyl type theorems."} \textit{Acta Mathematica Universitatis Comenianae} 86, no. 2 (2017): 345–356.

\bibitem{schmoeger1994operators} Schmoeger, Karel C., Herzog, G., and Zheng, Dechao. \textit{"On operators ${T}$such that $f ({T})$ is hypercyclic."} \textit{Studia Mathematica} 108, no. 3 (1994): 209--216.

\bibitem{stroe} Stroethoff, Karel, and Zheng, Dechao. \textit{"Toeplitz and Hankel operators on Bergman spaces."} \textit{Transactions of the American Mathematical Society} 329, no. 2 (1992): 773–794.

\bibitem{sun} Sun, Chenhui, and Cao, Xiaohong. \textit{"Criteria for the Property $(UW_E)$ and the a-Weyl Theorem."} \textit{Functional Analysis and Its Applications} 56, no. 3 (2022): 216–224.

\bibitem{sundberg} Sundberg, Carl, and Zheng, Dechao. \textit{"The spectrum and essential spectrum of Toeplitz operators with harmonic symbols."} \textit{Indiana University Mathematics Journal} (2010): 385–394.

\bibitem{Weyl} Weyl, Hermann Von. \textit{"Über beschränkte quadratische Formen, deren Differenz vollstetig ist."} \textit{Rendiconti del Circolo Matematico di Palermo (1884–1940)} 27, no. 1 (1909): 373–392.

\bibitem{wang} Wang, Chongchao, and Zhao, Xianfeng. \textit{"Weyl's theorem for Toeplitz operators with polynomial symbols on the harmonic Bergman space."} \textit{Journal of Mathematical Analysis and Applications} 495, no. 2 (2021): 124770.

\bibitem{yang} Yang, Lili, and Cao, Xiaohong. \textit{"Property $(UW_{\pi})$ for functions of operators and compact perturbations."} \textit{Mediterranean Journal of Mathematics} 19, no. 4 (2022): 163.

\end{thebibliography}
	\end{document}